                    	\def\version{1 August, 2012}	                          %

\documentclass[11pt, reqno]{amsart}
\usepackage{srcltx}
\usepackage[centertags]{amsmath}
\usepackage{amsthm, a4, latexsym, amssymb}
\usepackage[latin1]{inputenc}
\usepackage{color}
\usepackage{graphicx}
\usepackage{dsfont}

\def\@rmrk#1#2{\refstepcounter
    {#1}\@ifnextchar[{\@yrmrk{#1}{#2}}{\@xrmrk{#1}{#2}}}

\makeatletter\@addtoreset{equation}{section}\makeatother

\newfont{\bfit}{cmbxti10 scaled 2000}
\newfont{\biggi}{cmr12 scaled 2000}
\newtheorem{step}{STEP}

\newcommand{\bes}{\begin{step}}
\newcommand{\es}{\end{step}}

 \newcommand{\eps}{\varepsilon}

 \newcommand{\e}{{\rm e}}

 \newcommand{\esssup}{{\rm esssup}\,}
 \newcommand{\essinf}{{\rm essinf}\,}
 \newcommand{\R}{\mathbb{R}}
 \newcommand{\Z}{\mathbb{Z}}
 \newcommand{\N}{\mathbb{N}}

 \renewcommand{\d}{\,\mathrm{d}}
 
 \newcommand{\prob}{\mathbb{P}}
 \newcommand{\Prob}{{\rm Prob }}

 \renewcommand{\P}{\mathbb{P}}
 \newcommand{\E}{\mathbb{E}}

 \newcommand{\Ncal}{{\mathcal N}}
 \newcommand{\Ical}{{\mathcal I}}
 
 \newcommand{\Pcal}{{\mathcal P}}

 \newcommand{\Tcal}{{\mathcal T}}
 \newcommand{\Xcal}{{\mathcal X}}
 \newcommand{\Scal}{{\mathcal S}}

 \newcommand{\ssup}[1] {{\scriptscriptstyle{({#1}})}}
\def\1{{\mathchoice {1\mskip-4mu\mathrm l}      
{1\mskip-4mu\mathrm l}
{1\mskip-4.5mu\mathrm l} {1\mskip-5mu\mathrm l}}}

 \newcommand{\kommentar}[1]{}

\renewcommand{\subsection}{\secdef \subsct\sbsect}
\newcommand{\subsct}[2][default]{\refstepcounter{subsection}
\vspace{0.15cm}
{\flushleft\bf \arabic{section}.\arabic{subsection}~\bf #1  }
\nopagebreak\nopagebreak}
\newcommand{\sbsect}[1]{\vspace{0.1cm}\noindent
{\bf #1}\vspace{0.1cm}}

\newtheorem{theorem}{Theorem}[section]
\newtheorem{lemma}[theorem]{Lemma}

\newtheorem{Assumption}[theorem]{Assumption}

\theoremstyle{definition}
\newtheorem{remark}[theorem]{Remark}

\def\thebibliography#1{\section*{Bibliography}
  \list%
  {\arabic{enumi}.}
    {\settowidth\labelwidth{[#1]}\leftmargin\labelwidth
    \advance\leftmargin\labelsep
    \parsep0pt\itemsep0pt
    \usecounter{enumi}}
    \def\newblock{\hskip .11em plus .33em minus .07em}
    \sloppy                   
    \sfcode`\.=1000\relax}

\def\P{\prob}

\begin{document}
\title[Moment asymptotics for branching random walks in random environment]{\Large Moment asymptotics\\\medskip for branching random walks\\\medskip in random environment}
\author[Onur G\"un, Wolfgang K\"onig and Ozren Sekulovi\'c]{}
\maketitle
\thispagestyle{empty}
\vspace{-0.5cm}

\centerline{\sc By Onur G\"un\footnote{Weierstrass Institute Berlin, Mohrenstr.~39, 10117 Berlin, {\tt guen@wias-berlin.de} and {\tt koenig@wias-berlin.de}}, Wolfgang K\"onig$ $\footnotemark[1]$^,$\footnote{Institute for Mathematics, TU Berlin, Str.~des 17.~Juni 136, 10623 Berlin, Germany, {\tt koenig@math.tu-berlin.de}}, and Ozren Sekulovi\'c\footnote{Freie Universit\"at Berlin, Habelschwerdter Allee 45, 14195 Berlin, {\tt ozrens@t-com.me}}}
\renewcommand{\thefootnote}{}


\setcounter{footnote}{3}
\renewcommand{\thefootnote}{\arabic {footnote}}
\vspace{-0.5cm}
\centerline{\textit{Weierstrass Institute Berlin, TU Berlin and FU Berlin}}
\vspace{0.2cm}

\begin{center}
\version
\end{center}

\begin{quote}{\small }{\bf Abstract.} We consider the long-time behaviour of a branching random walk in random environment on the lattice  $\Z^d$. The migration of particles proceeds according to simple random walk in continuous time, while the medium is given as a random potential of spatially dependent killing/branching rates. The main objects of our interest are the annealed moments $\langle m_n^p \rangle $, i.e., the  $p$-th moments over the medium of the $n$-th moment over the migration and killing/branching, of the local and global population sizes. For $n=1$, this is well-understood \cite{GM98}, as $m_1$ is closely connected with the parabolic Anderson model. For some special distributions, \cite{A00} extended this to $n\geq2$, but only as to the first term of the asymptotics, using (a recursive version of) a Feynman-Kac formula for $m_n$.

In this work we derive also the second term of the asymptotics, for a much larger class of distributions. In particular, we show that $\langle m_n^p \rangle$ and $\langle m_1^{np} \rangle$ are asymptotically equal, up to an error $\e^{o(t)}$. The cornerstone of our method is a direct Feynman-Kac-type formula for $m_n$, which we establish using the spine techniques developed in \cite{HR11}.
\end{quote}

\setcounter{section}{0} 

\setcounter{tocdepth}{2}


\section{Introduction}

\noindent Random processes in random surroundings are under investigation for decades. Examples of such processes include (1) trajectories of random walks and Brownian motion in random environment with the focus on laws of large numbers and central limit theorems or even invariance principles, (2) heat equation and other partial differential equation systems in random potential with the focus on intermittent behaviour, (3) polymer measures and directed percolation in random  medium with the focus on free energies. In all these models, a rich phenomenology of asymptotic behaviours arises, which is not shared by the original model in non-random (homogeneous) surrounding, and in most of these models the research goes on, as many of the main features have not yet been properly understood. 

Another large class of random processes that develop striking properties in random surroundings is the class of  branching processes. Let us briefly describe some of the work that was carried out about these models. Branching discrete random walks on $\Z^d$ with time-space i.i.d.~offspring distributions were studied in the context of survival properties, global/local growth rates and diffusivity; and their connections to the directed polymers in random environment, see e.g.~\cite{BGK05,Y08,CY11}. Detailed analyses of recurrence/transience properties of discrete-time branching Markov chains with only space-dependent environment, which does not exhibit in general the the usual dichotomy valid for irreducible Markov chains, were carried out in \cite{CMP98, MP00, MP03, CP07, M08, BGK09, GMPV10}, to mention some. The main techniques in these studies relate these models to the better-known random walk in random environments, using the spectral properties of underlying Markov process and studying the embedded Galton-Watson processes in random environment.

In this paper, we study a {\it branching random walk in random environment (BRWRE)}, where the particles move around in space like independent random walks in continuous time, and the killing/branching takes place in sites with a random site-dependent rate. We are interested in the long-time asymptotics of the annealed moments of any order of the local and global population sizes. As was explained in \cite {GM90} for the case of first moments, this question stands in a close connection with the description of the intermittent behaviour of the main particle flow, i.e., its concentration behaviour in small islands. According to the best of our knowledge, this question for the higher moments has hardly been investigated for this model yet, the only example being \cite{A00}. In that paper, a deep relation between the moments of the BRWRE and the parabolic Anderson model is revealed and employed in order to analyse the annealed moments of the BRWRE, i.e., the  $p$-th moments over the medium of the $n$-th moment over the killing/branching and migration of the total and local population size.  It is the aim of the present paper to significantly increase the validity and the deepness of the results of \cite{A00} and to reveal the general mechanism that leads to the moment asymptotics. In contrast with \cite{A00}, we will be using probabilistic methods rather than PDE methods.

\subsection{Branching random walk in random environment}\label{sec-BRWRE}

\noindent Let us describe the model in more detail. The branching random environment on the lattice $\Z^d$ is a pair $\Xi =(\xi_0,\xi_2)$ of two independent i.i.d.~fields $\xi_0=(\xi_0(y))_{y\in\Z^d}$ and $\xi_2=(\xi_2(y))_{y\in\Z^d}$ of positive numbers. Indeed, $\xi_0(y)$ and $\xi_2(y)$ play the r\^{o}le of the rate of the replacement of a particle at $y \in \Z^d$  with 0 or 2 particles, respectively. For $n=0$, this is a killing, for $n=2$, this is a binary splitting. (See Section~\ref{sec-moregeneral} for more general branching mechanisms.)

The probability measure corresponding to $\Xi$ is denoted $\Prob$; expectation with respect to $\Prob$ will be written with angular brackets $\langle\cdot\rangle$. For a given realization of $\Xi$, the branching process with rate field $\Xi$  is now defined by determining that any particle located at a lattice site $y \in \Z^d$, is subject to the killing/branching defined by the rates $\xi_0(y)$ and $\xi_2(y)$, and additionally each particle performs a continuous-time random walk on $\Z^d$ with generator $\kappa\Delta$, where $\kappa>0$ is a parameter, and 
$$
    \Delta f(x)= \sum_{y \sim x} \bigl[f(y)-f(x)\bigr], \qquad
    \mbox{ for } x\in \Z^d,\, f\in\ell^2(\Z^d),
$$
is the standard lattice Laplacian. We write expectation with respect to a random walk with generator $\kappa\Delta$ starting from $x$ as $\P_x$ with corresponding expectation $\E_x$. We consider a localised initial condition, i.e., at time $t=0$, there is a single particle at some site $x\in\Z^d$. Probability and expectation w.r.t.~the migration, branching and killing of the BRWRE are denoted by $P_x$ and $ E_x$, respectively, for fixed medium $\Xi$. 

The description of the dynamics of the population is as follows. If a particle is at some time at some site $y$, then during a small time interval of length $h$, with probability $\kappa h + o(h)$ it moves to a neighbouring site chosen uniformly at random, with probability $\xi_2(y) h + o(h)$ it dies and is replaced by two descendant particles, and with probability $\xi_0(y)h + o(h)$ it is killed without producing any offspring. Finally, with probability $1-(\kappa+\xi_2(y)+\xi_0(y))h + o(h)$ the particle experiences no changes during the whole time interval of length $h$.

Let $\eta(t,y)$ be the number of particles at time $t\in[0,\infty)$ at $y\in\Z^d$, and let $\eta(t)=\sum_{y\in{\Z^d}} \eta(t,y)$ be the total population size at time $t$. The main objects of interest in this paper are the quenched moments 
\begin{equation}
m_n(t,x,y)={ E}_x[\eta(t,y)^n]\qquad\mbox{and}\qquad m_n(t,x)={ E}_x[\eta(t)^n],\qquad n\in\N,
\end{equation}
i.e., the expected $n$-th powers of the local and global particle numbers, where the expectation is taken only over the migration and the killing/branching, for frozen killing/branching rates $\Xi$. Note that, for $n=1$, $m_1(t,x)$ is equal to the sum of $m_1(t,x,y)$ over $y\in\Z^d$, but such a relation is not valid for $n\geq2$. 

It will be the main purpose of the present paper to analyse the large-$t$ asymptotics of the $p$-th moments of  $m_n(t,x)$ and of $m_n(t,x,y)$, taken over the medium $\Xi$.

\subsection{Connection with the parabolic Anderson model}\label{sec-PAMconn}

\noindent It is a fundamental knowledge in the theory of branching processes that the expected particle number satisfies certain partial differential equation systems. In our case, the characteristic system reads as follows. Put 
\begin{equation*}
\xi=\xi_2-\xi_0,
\end{equation*}
and fix $y\in\Z^d$, then, (under certain integrability conditions, see \cite{GM90}) for fixed localised initial condition $m_1(0,\cdot,y)=\delta_y(\cdot)$, the map $(t,x)\mapsto m_1(t,x,y)$ is the unique positive solution to the Cauchy problem for the heat equation with potential $\xi$, i.e.,
    \begin{eqnarray}
    \frac{\partial}{\partial t} m_1(t,x,y) & = & \kappa\Delta m_1(t,x,y) + \xi(x) m_1(t,x,y),\label{PAM1}
    \qquad \mbox{ for } (t,x)\in(0,\infty)\times \Z^d.
    \end{eqnarray}
Similarly, the map $(t,x)\mapsto m_1(t,x)$ is the unique positive solution of (\ref{PAM1}) with delocalized initial condition $m_1(t,\cdot)\equiv 1$.

The interesting feature in our case is that the potential $\xi$ is random, and here \eqref{PAM1} is often called the {\it parabolic Anderson model}. In fact, the operator $\kappa\Delta+\xi$ appearing on the right-hand side is called the {\it Anderson operator}; its spectral properties are well-studied in mathematical physics. Equation~\eqref{PAM1} describes a random mass transport through a random field of
sinks and sources, corresponding to lattice points $z$ with $\xi(z)<0$ and $\xi(z)>0$, respectively. We refer the reader to \cite{GM90}, \cite{M94} and \cite{CM94} for more background and to \cite{GK04} for a survey on mathematical results. We see two competing effects: the diffusion mechanism (Laplacian) tends to make the field $m_1$ flat, and the local growth (potential) tries to make it irregular. 

Furthermore, it is also widely known since long \cite{GM90} that $m_1$ admits a representation in terms of the {\it Feynman-Kac formula}:
    \begin{equation}
    m_1(t,x,y)=\E_x \Bigl[\exp\Bigl\{\int_0^t\xi(X_s)\, \d s\Bigr\}\delta_y(X_t)\Bigr],\qquad (t,y)\in[0,\infty)\times\mathbb Z^d,\label{PAM1FKlocal}
    \end{equation}
and the same formula without the last indicator for $m_1(t,x)$, where $(X_s)_{s\in[0,\infty)}$ denotes a simple random walk with generator $\kappa\Delta$. Note that $m_1$ depends only on the difference $\xi$ of $\xi_2$ and $\xi_0$. 

The asymptotics of the moments of $m_1$ were analysed in \cite{GM98} for the interesting special case that the distribution of $\xi$ lies in the vicinity of the so-called {\it double-exponential distribution} with parameter $\rho\in(0,\infty)$,
    \begin{equation}
    \Prob(\xi(x)>r)=\exp\{-\e^{r/\rho}\}, \qquad r\in(0,\infty). \label{DExpTail}
    \end{equation}
The precise assumption on $\xi$ can be written down in terms of the logarithmic moment generating function 
\begin{equation}
H(t)=\log\langle \e^{t\xi(0)}\rangle,
\end{equation}
which is assumed to be finite for any $t>0$.

\begin{Assumption}\label{DoubleExp} There exists $\rho\in[0,\infty]$ such that
    \begin{equation}\label{AssumptionHl}
    \lim_{t\rightarrow\infty}\frac{H(ct)-cH(t)}{t}=\rho c\log c,\qquad c\in(0,1). 
    \end{equation}
\end{Assumption}
Under this assumption, it is proven in \cite{GM98} that, for any $x\in \Z^d$, as $t  \rightarrow \infty$,
    \begin{equation}\label{MomentGM98}
    \langle m_1^p(t,x) \rangle = \e^{H(pt)}\,\e^{ - 2d\kappa \chi (\rho/\kappa) pt + o(t)},\qquad p\in\N, 
    \end{equation}
where $\chi$ is defined as
    \begin{equation}\label{chi}
    \chi (\rho) = \frac 12\inf_{\mu \in \mathcal{P}(\Z)}{[\Scal(\mu)+\rho \Ical(\mu)]}.  
    \end{equation}
Here $\mathcal{P}(\Z)$ denotes the space of probability measures on $\Z$, and the  functionals $ \Scal,\Ical\colon \mathcal{P}(\Z) \rightarrow \R_{+}$ are given by
    \begin{equation}
    \Scal(\mu) = \sum_{x \in \Z} \big(\sqrt{\mu(x+1)}-\sqrt{\mu(x)}\big)^2   \qquad  \mbox{and}\qquad \Ical(\mu) = -\sum_{x \in \Z} \mu(x)\log \mu(x).
    \end{equation}
We have $0 < \chi (\rho)< 1$ for $\rho\in(0,\infty)$ and $\chi(0)=0$ and $\lim_{\rho \rightarrow \infty}\chi(\rho)=\chi(\infty)=1$. The right-hand side of \eqref{MomentGM98} is also equal to the moments of $m_1(t,x,y)$ for any fixed $x, y\in\Z^d$, as is seen from an inspection of the proof (see Remark 1.3 in \cite{GM98}). Also note that $H(t)\gg 2d\kappa \chi (\rho/\kappa) pt$ for large $t$, that is, asymptotically the first term on the right-hand side of \eqref{MomentGM98} is much larger than the second term.

Observe that the $p$-th moments of $m_1$ at time $t$ behave  like the first moment at time $tp$, up to the precision of \eqref{MomentGM98}. This can be easily guessed from a standard eigenvalue expansion for $m_1(t,x)$ in terms of the eigenvalues and eigenfunctions of $\kappa\Delta$ in large $t$-dependent boxes with zero or periodic boundary condition; in fact, $m_1(t,x)$ is roughly equal to $\e^{t\lambda_1(t)}$, where $\lambda_1(t)$ is the principal one. Then, obviously, $m_1^p(t,x)$ is roughly equal to $\e^{tp\lambda_1(t)}$.

\subsection{Moments of the BRWRE}\label{sec-momentsBRWRE}

\noindent Let us now turn to the main object of the present paper, the moments of $m_n$ for $n\geq 2$. We can formulate our main result. Recall our assumptions from the beginning of Section~\ref{sec-BRWRE}. We will also suppose that the branching rate $\xi_2(0)$ satisfies Assumption~\ref{DoubleExp}. In the case $\rho=\infty$, we will need an extra assumption to avoid too large a growth of $H_2(t)$:
\begin{Assumption}\label{momentassump} For any $k\in\N$,
\begin{equation}
\big\langle \xi_2(0)^k \,\e^{\xi_2(0) t}\big\rangle\leq \langle \e^{\xi_2(0) t}\rangle\e^{o(t)}\;\; \text{ as } t\to\infty.
\end{equation}
\end{Assumption}

\begin{theorem}[Moments of the BRWRE]\label{thm-Main} Suppose that the logarithmic moment generating function $H_2$ of $\xi_2(0)$ satisfies Assumption~\ref{DoubleExp} and, in the case $\rho=\infty$, $\xi_2$ also satisfies Assumption~\ref{momentassump}. Fix $x\in\Z^d$, the starting site of the branching process. Then, for any $p, n\in\mathbb{N}$, as $t\to\infty$,
    \begin{equation}\label{Moment_n_p}
    \langle m_n^p(t,x) \rangle = \exp\Big({H(npt)}\,- 2d\kappa \chi (\rho/\kappa) npt + o(t)\Big).
    \end{equation}
The same asymptotics holds true for $\langle m_n^p(t,x,y)\rangle$ for any $y\in\Z^d$.
\end{theorem}

Note that the logarithmic moment generating function $H_0$ of $-\xi_0(0)$ has asymptotics $\frac 1t H_0(t)\to -\essinf(\xi_0(0))\in(-\infty,0]$ as $t\to\infty$. Therefore, by independence of $\xi_2$ and $\xi_0$, the logarithmic moment generating function $H$ of $\xi_2(0)-\xi_0(0)$ also satisfies Assumption~\ref{DoubleExp}, and this is crucial for the validity of Theorem~\ref{thm-Main}.

In particular, Theorem~\ref{thm-Main} says that the $p$-th moments of $m_n$ at time $t$ are equal to the first moment of $m_1$ at time $tpn$, up to the precision in \eqref{Moment_n_p}, i.e.,
    \begin{equation}\label{Moment_n_p_heur}
    \langle m_n^p(t,x) \rangle= \langle m_1^{np}(t,x) \rangle \e^{o(t)}=\langle m_1(tnp,x) \rangle \e^{o(t)}, \qquad  t\rightarrow\infty.
    \end{equation}

This fact is not so easy to understand as for the case $n=1$, see above. However, see Section~\ref{sec-Explanation} for some heuristic remarks. The proof of Theorem~\ref{thm-Main} is in Section~\ref{sec-proofThm}. 

The main tool of our proof is a Feynman-Kac-type formula for $m_n$, which we will derive in Section~\ref{sec-FKform}, see Theorem~\ref{thrmmoment}. We are going to use probabilistic tools from the theory of branching processes, the main input coming from the {\it many-to-few lemma} of \cite{HR11}.

In  \cite{A00} there was a weaker version of \eqref{Moment_n_p} derived; actually only the first term $\e^{H(npt)}$, and this only for the rather restricted case of a Weibull distribution, where $H(t)\sim C t^\alpha$ for some $C\in(0, \infty)$ and $\alpha\in(1,\infty)$, a subcase contained in Assumption~\ref{DoubleExp} in $\rho=\infty$. On the other hand, they drop the assumption of independence and only assume spatial homogeneity of $\Xi$. However, this result does not explain the spatial structure of the peaks of the moments of the population size, an information that is contained in the second term, as was discussed at length in \cite{GM98}. The proof in \cite{A00} is based on the fact that $m_n$ is the solution to an inhomogeneous Cauchy problem, where the inhomogeneity is a linear combination of products of $m_1,\dots,m_{n-1}$. Furthermore, they derived from this a Feynman-Kac formula for $m_n$, which depends on that inhomogeneity and is therefore of recursive type. This made it rather difficult to identify the second term of the asymptotics. In contrast, we first derive a direct version of a Feynman-Kac-type formula in Theorem~\ref{thrmmoment} and are then able to find the logarithmic asymptotics of the moments in much higher precision.

\subsection{More general branching}\label{sec-moregeneral}

\noindent Our Theorem~\ref{thm-Main} is formulated only for the special case of binary branching, but it can straightforwardly be extended to more general branching mechanisms, subject to additional conditions. Indeed, assume that the branching random environment is a family $\Xi =(\xi_k)_{k\in\N_0}$ of i.i.d.~fields $\xi_k=(\xi_k(y))_{y\in\Z^d}$ of positive numbers. Then $\xi_k(y)$ is the rate for replacement of a particle at $y$ by precisely $k$ new particles, i.e., a splitting into $k$ particles. To exclude trivialities, we put $\xi_1(y)=0$ for any $y$.  The family $(\xi_k)_{k\in\N_0}$ is not assumed to be i.i.d. Indeed, we at least have to assume that the field
\begin{equation}\label{xidef}
 \xi(y)=\sum_{k=0}^\infty (k-1) \xi_k(y)
\end{equation}
is well-defined (i.e., absolutely convergent) almost surely. One possible choice could be $\xi_k=\xi p_k$ with some probability distribution $(p_k)_{k\in \N_0}$ and some positive i.i.d.~field $\xi$. 

Then, under the assumption that $\sum_{k\in\N}k^n\xi_k(y) < \infty$ almost surely (e.g., if $\xi_k\equiv 0$ for all sufficiently large $k$), our Feynman-Kac-type formula for $m_n$ in Theorem~\ref{thrmmoment} below extends to this more general setting, see Remark~\ref{MomFormGeneral}. Furthermore, under suitable conditions on the moments of $\sum_{k\in\N}k^n\xi_k(y) $, also the proof of Theorem~\ref{thm-Main} in Section~\ref{sec-proofThm} can be easily extended to this situation. In order to avoid cumbersome formulas, we abstained from writing down the details.

\subsection{Discussion}\label{sec-Explanation}

\noindent Let us explain why the moment asymptotics of $m_n^p(t,x)$ are equal to the ones of $m_1^{pn}(t,x)$, see \eqref{Moment_n_p_heur}. We do this for $n=2$ and $p=1$. Note that, according to Theorem~\ref{thrmmoment} below, $m_2=m_1+\widetilde{m}_2$, where
\begin{equation}\label{m2formula}
\widetilde m_2(t,x) =\int_0^t\E_x \Bigl[\exp\Bigl\{\int_0^s\xi(X_r)\, \d r + \int_s^t\xi(X'_r)\, \d r + \int_s^t\xi(X''_r)\, \d r\Bigr\} 2\xi_2(X_s)\Bigr]\, \d s,
\end{equation}
where $(X_r)_{r\in[0,s]}$ and $(X'_r)_{r\in [s,t]}$ and $(X''_r)_{r\in [s,t]}$ are independent simple random walks, given $X_s$, with generator $\kappa\Delta$, starting at $X_0=x$, and $X'_s=X''_s=X_s$. In other words, these three random walks constitute a branching random walk with precisely one splitting at time $s$. The first part in the decomposition $m_2=m_1+\widetilde{m}_2$, corresponds to absence of splitting, and the second one to precisely one splitting.

Let us consider the behaviour of the moments of $\widetilde m_2$ as $t \rightarrow \infty$. The first observation is that the term $2\xi_2(X_t)$ should have hardly any influence. This is due to Assumption~\ref{momentassump}, which rules out cases of extreme growth of $H_2(t)$. One can expect from \eqref{MomentGM98} that the leading term of the expectation on the right-hand side of \eqref{m2formula} should be $\e^{H(2t-s)}$, corresponding to the total time $s+(t-s)+(t-s)$ that the three random walks spend in the random environment. Since $H(t)\to\infty$, this is clearly maximal for $s\approx 0$. Hence, the Laplace method gives that the main contribution comes from $s\approx 0$. Hence, the contribution comes mainly from a product of expectations over two i.i.d.~copies $(X'_r)_{r\in [0,t]}$ and $(X''_r)_{r\in [0,t]}$, i.e., from a term $\approx m_1^2(t,x)$.

In other words, it is favourable for the branching random walk to split as soon as possible into two copies and to travel through the environment with these copies for a long time. The deeper reason for this is that the potential $\xi$ assumes extremely high values in some part of the space, where the two copies collect much of them. This effect seems to be present as soon as $\esssup(\xi(0))$ is positive, and it should be turned into its opposite if $\esssup(\xi(0))$ is negative. More precisely, for such potentials, we expect that $\langle m_n(t,x)\rangle\approx \langle m_1(t,x)\rangle$. We expect that, for all four classes of potentials in the classification made in \cite{HKM06}, a version of Theorem~\ref{thm-Main} can be deduced from Theorem~\ref{thrmmoment}.

\section{Feynman-Kac-type formula for $m_n$ via spine techniques}\label{sec-FKform}

\noindent In this section, we derive a Feynman-Kac-type formula for $m_n$, almost surely with respect to the branching rates $\xi_2$ and killing rates $\xi_0$. Our main result of this section appears in Theorem~\ref{thrmmoment} below. We will use the spine techniques of \cite{HR11}. This requires the introduction of a branching random walk (BRW) in $\Z^d$ with time interval $[0,t]$ with up to $n-1$ splitting events. In order to express this BRW, we will need the following ingredients.
\begin{enumerate}
 \item a tree that expresses the branching structure,
\item an ordering of the splitting sites of the tree to express their order in time,
\item a time duration attached to each bond, 
\item an expectation over a simple random walk bridge attached to each bond.
\end{enumerate}
In order to keep the notation simple, we restrict to binary branching. See Remark~\ref{MomFormGeneral} for more general branching mechanisms. 

We need some notation from the theory of trees. Let $G=(V,E)$ be a finite graph with $V$ the set of vertices and  $E$ the set of edges. $G$ is a {\it tree} if it is simple, connected and has no cycles. Let us assume that $G$ is a {\it rooted tree}, i.e., a tree with a {\it root} $\emptyset\in V$. This induces  a natural ordering of vertices, namely, for $u,v\in V$ we say that $u\preceq v$ if the unique path from the root to $v$ contains $u$. In particular, if $(u,v)\in E$ then either $u\preceq v$ or $v\preceq u$. We hence may assume that $E$ is a directed tree, i.e., $E$ contains only edges $(u,v)$ with $u\preceq v$, in which case we call $u$ the {\it parent} of $v$ and $v$  a child of $u$. Note that, except the root, each vertex has a unique parent. We call a vertex a {\it leaf} if it has no children. We call $G$ a {\it rooted binary tree} if each vertex has at most two children. We distinguish binary trees by labelling the children of each vertex as the left child and the right child.

By $\mathcal{T}_k$ we denote the set of finite rooted binary trees with $k+1$ leaves, such that the root has precisely one child and every other vertex has precisely two children, except for the leaves. Note that $\mathcal T_0$ consists of one tree only, which consists of the root, a leaf and an edge going from the root to the leaf. Furthermore, put $\mathcal{T}=\bigcup_{k\in\N_0}\mathcal T_k$. For a tree in $\mathcal{T}$ we call the vertices other than the root and the leaves {\it splitting vertices}. Note that a tree in $\mathcal T_k$ has precisely $k$ splitting vertices. For $T=(V,E)\in\mathcal T$, we denote by $S$ the set of its splitting vertices and by $L$ the set of its leaves; hence $V=\{\emptyset\}\cup S\cup L$, $\#S=k$ and $\#L=k+1$. We write $T=(\emptyset,S,L,E)$. See Figure \ref{exampletrees} for two representatives from $\mathcal{T}_3$.

\begin{figure} [ht]
\begin{center}
    \centering
    \input{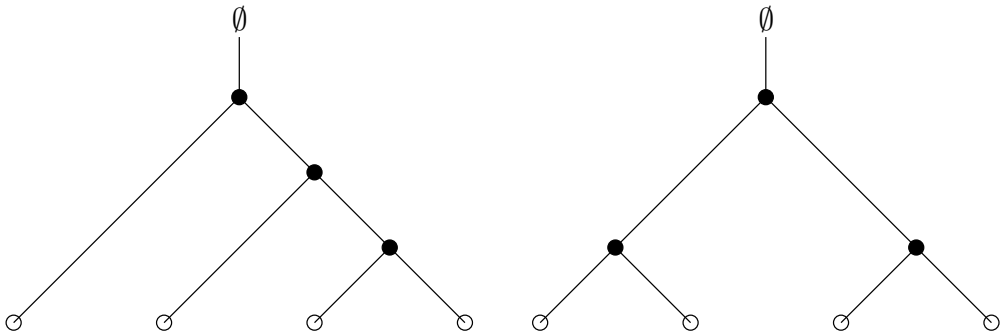}
  \end{center}
\caption{Two trees in $\mathcal{T}_3$. The empty circles represent the leaves $L$ and the full circles represent the branching vertices $S$.}
\label{exampletrees}
    \end{figure}

Now we equip trees with numberings. For  $k\in\mathbb{N}_0$ and $T=(\emptyset,S,L,E)\in\mathcal{T}_k$ let $I\colon \{\emptyset\}\cup S\to \{0,1,2,\dots,k\}$ be a bijection. We call $I$ a {\it monotonous numbering} of $T$ if $I(\emptyset)=0$ and $I(s_1)<I(s_2)$ for any $s_1,s_2 \in S$ with $(s_1,s_2)\in E$. We extend $I$ to $L$ by setting $I(l)=k+1$ for any leaf $l\in L$. See Figure \ref{exampleorderedtrees} for an example. The set of monotonous numberings of $T$ is denoted by $\mathcal{N}(T)$. 

\begin{figure} [ht]
\begin{center}
    \centering
    \input{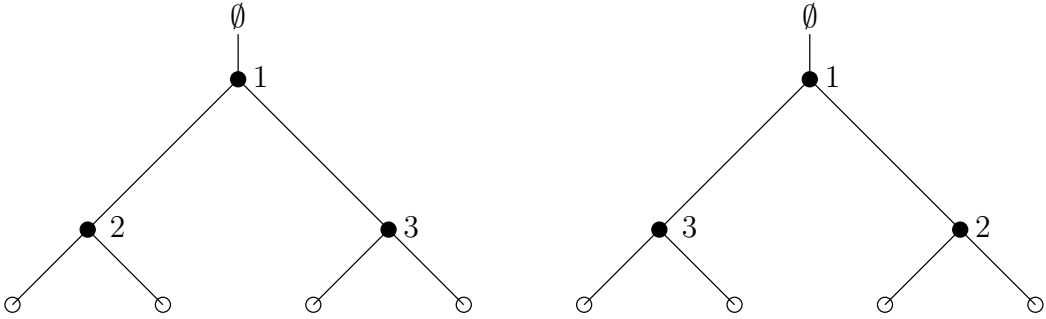}
  \end{center}
\caption{The only two possible monotonous numberings for the tree on the right of Figure \ref{exampletrees}. The left tree there admits only one such numbering.}
\label{exampleorderedtrees}
    \end{figure}

Now we equip numbered trees with times. For $k\in\mathbb{N}_0$ and $t>0$, denote by $Z_k(t)$ the set of time vectors
\begin{equation}
\widehat{t}=(t_0,\dots,t_{k+1}),\qquad\mbox{where}\qquad 0=t_0<t_1<\cdots<t_k<t_{k+1}=t.
\end{equation}
Let us fix a tree $T\in \mathcal{T}_k$, an ordering $I\in\mathcal{N}(T)$ and a time vector $\widehat t\in Z_k(t)$. For $b=(u,v)\in E$, we denote by 
$$
Y^{\ssup{b,\widehat t}}=\big(Y_r^{\ssup{b,\widehat t}}\colon r\in [t_{I(u)},t_{I(v)}]\big)
$$ 
a continuous-time simple random walk on $\Z^d$ with generator $\kappa \Delta$, starting from zero. We assume that the collection $(Y^{\ssup{b,\widehat t}})_{b\in E}$ is independent. We consider $Y^{\ssup{b,\widehat t}}$ as the segment of a branching random walk with parent $u$ and child $v$ that arises from a splitting event at time $t_{I(u)}$, considered until the next splitting event at time $t_{I(v)}$.

Now we compose all these segments of simple random walks according to the tree and define the  BRW on $[0,t]$ with precisely $k$ splits. Fix the starting site $x\in\Z^d$ of the branching process. For a leaf $l\in L$ let $\emptyset=u_0,\dots,u_{j-1},u_j=l$ be the vertices visited by the unique path from $\emptyset$ to the  leaf $l$, and $b_i=(u_{i-1},u_i)$ the corresponding bonds, where $j\in\mathbb{N}$. Then we define the continuous-time random walk $X^{\ssup l}=(X^{\ssup l}_r)_{r\in[0,t]}$ by
\begin{equation}
X_r^{\ssup l}:=x+\sum_{m=1}^{i-1} Y_{t_{I(u_{m})}}^{\ssup{b_m,\widehat t}} + Y_r^{\ssup{b_i,\widehat t}},\qquad r\in [t_{I(u_{i-1})},t_{I(u_i)}], i\in\{1,\dots,j\}.
\end{equation}
Note that the collection of the random walks $(X^{\ssup l})_{l\in L}$ is consistent in the sense that, for any leaves $l$ and $l'$, the paths of $X^{\ssup l}$ and $X^{\ssup {l'}}$ coincide up to the time $t_{I(\widetilde u)}$ of the vertex $\widetilde u$ where the tree path $\emptyset \to l$ splits from the path $\emptyset\to l'$; afterwards they are independent given the site $X^{\ssup l}_{t_{I(\widetilde u)}}=X^{\ssup {l'}}_{t_{I(\widetilde u)}}$. The separate pieces of the BRW between subsequent splits are denoted by 
$$
X^{\ssup {u,v}}=(X^{\ssup {u,v}}_r)_{r\in[t_{I(u)},t_{I(v)}]}=(X^{\ssup l}_r)_{r\in[t_{I(u)},t_{I(v)}]},\qquad (u,v)\in E,
$$
where $l\in L$ is any leaf such that the bond $(u,v)$ lies on the unique path from $\emptyset$ to $l$. Because of the above consistency property of $(X^{\ssup l})_{l\in L}$, the value does not depend on the choice of $l$. The collection of all the path pieces $X^{\ssup {u,v}}$ with $(u,v)\in E$ is consistent in the sense that $X^{\ssup {u,v}}_{t_{I(v)}}=X^{\ssup{v,u'}}_{t_{I(v)}}$ for any edges $(u,v)$ and $(v,u')$. See Figure~\ref{brwontree} for an example.

\begin{figure} [ht]
\begin{center}
    \centering
    \input{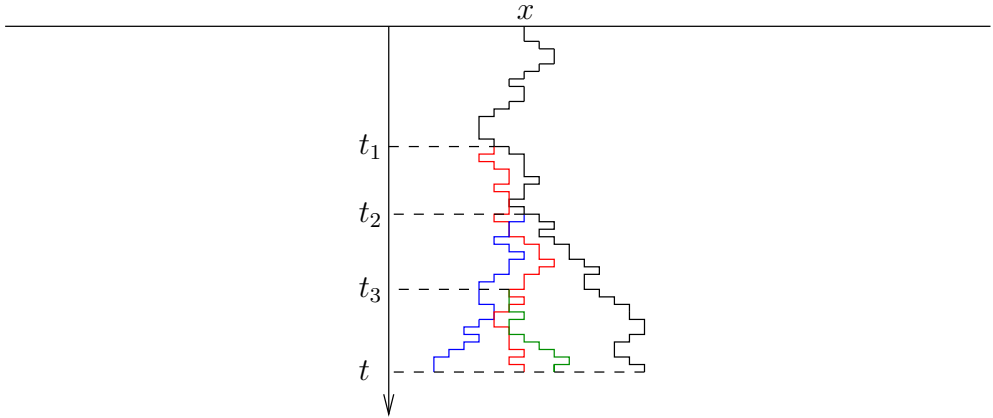}
  \end{center}
\caption{An example of a BRW corresponding to the monotonously numbered tree on the right of Figure \ref{exampleorderedtrees}. }
\label{brwontree}
    \end{figure}

Expectation with respect to the collection $(X^{\ssup l})_{l\in L}$ will be denoted by $\E^{\ssup{T,I,\widehat t}}_{x}$. For $y\in\Z^d$, we abbreviate
\begin{equation}
\Phi_x(T,I,t,y):=\int_{Z_k(t)} \d\widehat t\, \E^{\ssup{T,I,\widehat t}}_{x} \Big[\exp\Big(\sum_{(u,v)\in E}\;\;\int^{t_{I(v)}}_{t_{I(u)}} \xi(X^{\ssup {u,v}}_r)\,\d r \Big) \Big(\prod_{v\in S}\xi_2(X^{\ssup {u,v}}_{t_{I(v)}})\Big)\sum_{l\in L}\1\{X^{\ssup l}_{t}=y\}\Big], 
\end{equation}
where in the product $u$ is the parent of $v$. Furthermore, we define 
\begin{equation}\label{PhixTIdef}
\Phi_x(T,I,t):=\int_{Z_k(t)} \d\widehat t\, \E^{\ssup{T,I,\widehat t}}_{x} \Big[\exp\Big(\sum_{(u,v)\in E}\;\;\int^{t_{I(v)}}_{t_{I(u)}} \xi(X^{\ssup {u,v}}_r)\,\d r \Big) \Big(\prod_{v\in S}\xi_2(X^{\ssup {u,v}}_{t_{I(v)}})\Big)\Big].
\end{equation}

Finally, we define sequence of numbers $c_{k,n}$ for $n\in\mathbb{N}$ and $k=0,\dots,n-1$ by setting $c_{0,n}=1$ for all $n\in\mathbb{N}$ and by the recursive relation
\begin{equation}\label{combid}
c_{k,n}=\sum_{i=1}^{n-k}\binom{n}{i}c_{k-1,n-i},\qquad k=1,\dots,n-1.
\end{equation}

Now we can state the main theorem of this section, which gives us a Feynman-Kac-type formula for the functions $m_n$.

\begin{theorem}\label{thrmmoment} For $n\in\mathbb{N}$ and $x,y\in\mathbb{Z}^d$, we have
    \begin{equation}\label{MomentFormula}
    m_n(t,x)=\sum_{k=0}^{n-1} \sum_{T \in \mathcal{T}_k}\sum_{I \in \mathcal{N}(T)}c_{k,n}\Phi_x(T,I,t),
    \end{equation}
and the same formula for $m_n(t,x,y)$ with $\Phi_x(T,I,t)$ replaced by $\Phi_x(T,I,t,y)$.
\end{theorem}

\begin{proof} We denote by $N(t)$ the set of particles alive at time $t$. For a particle $u\in N(t)$ let $\sigma_u$ and $\tau_u$ denote the birth and death time of $u$, respectively. We put $\sigma_u(t)=\sigma_u\wedge t$ and $\tau_u(t)=\tau_u\wedge t$. If $u\in N(t)$ we write $Z^{\ssup u}_s$ for the position of the unique ancestor of $u$ alive at time $s\in[0,t]$. If $u$ has no children we say that $Z_s^{\ssup u}$ is at the graveyard state, $\partial$, for any $s\geq \tau_u$.

Specialising  \cite[Section 2]{HR11} to our situation, we define a new branching process by imposing the following rules:
\begin{enumerate}
\item We start with one particle at $x$, which carries $n$ marks (and their positions) $1,2,\dots,n$.

\item We think of each of the marks $1,2,\dots,n$ as a spine and denote by $\zeta_t^{\ssup i}$ the position of the whichever particle that carries the mark $i$ at time $t$.

\item Particles diffuse as under $P_x$, i.e., as independent continuous-time random walks with generator $\kappa\Delta$.

\item A particle at position $y$ carrying $j$ marks branches at rate $2^j{\xi_2(y)}$ and is replaced with two new particles.

\item At such a branching event of a particle carrying $j$ marks, each mark chooses independently and uniformly at random one of the two particles to follow.

\item Particles not carrying any marks behave as under $P_x$.
\end{enumerate}

We write $\mathbb{Q}_x^{\ssup n}(\cdot)$ for the corresponding probability measure and $\mathbb{Q}_x^{\ssup n}[\cdot]$ for the corresponding expectation. We call the collection of particles that have carried at least one mark up to time $t$ the {\it skeleton} at time $t$ and write ${\rm skel}(t)$. We define $D(v)$ as the number of marks carried by a particle $v$. Let us abbreviate
\begin{equation}
A(t)=\exp\Big(\sum_{v\in{\rm skel}(t)}\int_{\sigma_v(t)}^{\tau_v(t)}\Big((2^{D(v)}-1)\xi_2(Z^{\ssup v}_r)-\xi_0(Z_r^{\ssup v})\Big)\,\d r\Big).
\end{equation}
We now apply the many-to-few lemma \cite[Lemma 3]{HR11} for $Y=1$ and $\zeta\equiv 1$ and obtain
\begin{equation}\label{manytofew}
m_n(t,x)=\mathbb{Q}_x^{\ssup n}[A(t)]\qquad\mbox{and}\qquad m_n(t,x,y)=\mathbb{Q}_x^{\ssup n}\Big[A(t)\sum_{v\in {\rm skel}(t)}\mathds{1}\{Z_t^{\ssup v}=y\}\Big].
\end{equation}

\newcommand{\bq}{\overline{\mathbb{Q}}_x^{\ssup n}}

Note that the spine trajectory does not undergo a splitting at a branching event, if all the marks choose the same child to follow. Hence, we only want to consider splitting events that not all the marks choose the same particle to follow. Note that the probability of such event when a particle carrying $j$ marks branches is $1-2^{-j+1}$. Then the rate of such branching events for particles carrying $j$ marks at position $y$ is $(2^j-2)\xi_2(y)$. Accordingly, we define a measure $\bq$ by changing the items (iv) and (v) in the above description of $Q_x^{\ssup n}$ by
\begin{enumerate}
\item[($\overline{iv}$)]  A particle at position $y$ carrying $j$ marks branches at rate $(2^j-2){\xi_2(y)}$ and is replaced with two new particles.

\item[($\overline{v}$)]  At a branching event of a particle carrying $j$ marks choose uniformly at random one of the two particles to follow  conditioned on for each new particle there is at least one mark following it.
\end{enumerate}

Note that \eqref{manytofew} is still valid when $\mathbb{Q}_x^{\ssup n}$ is replaced by $\bq$. 

We only prove \eqref{MomentFormula} since the proof of the formula for the moments of $m_n(t,x,y)$ is done exactly in the same way. We proceed the proof by the method of strong induction. For $n=1$, (\ref{MomentFormula}) is immediate. Now assume that (\ref{MomentFormula}) holds for $n$ replaced by any $i\in\{1,\dots,n-1\}$, and we prove that it is also true for $n$. 

We start from the first formula in \eqref{manytofew} and integrate over all values of the time, $T$, of the first branching event under $\overline{Q}_x^{\ssup n}$ and over all possible branchings. The conditional distribution of $T$ given $\zeta^{\ssup 1}$ is given by
\begin{equation}
\bq\big(T>t\,\big|\,\zeta^{\ssup 1 }\big)
=\exp\Big(-\int_0^t (2^n-2)\xi_2(\zeta^{\ssup 1}_r)\,\d r\Big).
\end{equation}
On the event $\{T>t\}$, we have ${\rm skel}(t)=\{\emptyset\}$, $\sigma_\emptyset(t)=0$, $\tau_\emptyset(t)=t$, $D(\emptyset)=n$ and $Z^{\ssup{\emptyset}}=\zeta^{\ssup 1}$. Hence, we have
\begin{equation}\label{inter1}
\overline{\mathbb{Q}}_x^{\ssup n}\Big[A(t)\1_{\{T>t\}}\,\Big|\,\zeta^{\ssup 1 }\Big]
=\exp\Big(\int_0^t \xi(\zeta^{\ssup 1}_r)\,\d r\Big).
\end{equation}
Integrating this with respect to $\bq$, we get
\begin{equation}\label{inter2}
\begin{aligned}
\bq\big[A(t)\1_{\{T>t\}}\big]&=\bq\Big[\text{l.h.s.~of }(\ref{inter1})\Big]=\bq\left[\exp\left(\int_0^t \xi(\zeta^{\ssup 1}_r)\,\d r\right)\right]\\
&=\E_x\left[\exp\left(\int_0^t \xi(X_r)\,\d r\right)\right],
\end{aligned}
\end{equation}
since any spine follows a simple random walk with generator $\kappa\Delta$. This is the term that corresponds to $k=0$ in the sum in (\ref{MomentFormula}). Similarly we can calculate the conditional density of $T$ as
\begin{equation}\label{proofeq1}
\overline{\mathbb{Q}}_x^{\ssup n}\big(T\in \d t_1\,\big|\,\zeta^{\ssup 1}\big)=\exp\left(-\int_0^{t_1} (2^n-2)\xi_2(\zeta^{\ssup 1}_r)\,\d r\right) (2^n-2)\xi_2(\zeta^{\ssup 1}_{t_1})\,\d t_1,\qquad t_1>0.
\end{equation}

Let $B_{l,n-l}$ be the event that at the branching time $T$, $l$ marks follow the first child of $\emptyset$ and $n-l$ marks follow the second child. Then it is clear that
\begin{equation}\label{proofeq2}
\bq\big(B_{l,n-l}\big)=\binom{n}{l}\frac{ 1}{2^n-2},\qquad l=1,\dots,n-1. 
\end{equation}
So for $l=1,\dots,n-1$, by \eqref{manytofew}, we have, using the Markov property at time $t_1 \in[0,t]$,
\begin{equation}\label{proofeq3}
\begin{aligned}
\bq&\big[A(t)\,\big|\, B_{l,n-l},T=t_1,\zeta^{\ssup 1} \big]\\
&=\exp\Big(\int_0^{t_1}\Big\{(2^n-1)\xi_2(\zeta^{\ssup 1}_r)-\xi_0(\zeta_r^{\ssup 1})\Big\}\,\d r\Big)\,m_l(t-t_1,\zeta^{\ssup 1}_{t_1})\,m_{n-l}(t-t_1,\zeta^{\ssup 1}_{t_1}).
\end{aligned}
\end{equation}
Hence, using (\ref{proofeq3}), (\ref{proofeq1}), (\ref{proofeq2}) and the independence of the splitting time and the number of offsprings, we get
\begin{equation}\label{proofeq4}
\begin{aligned}
\int_0^t&\sum_{l=1}^{n-1}\bq\Big[A(t)\1_{B_{l,n-l}},T\in \d t_1\,\Big|\,\zeta^{\ssup 1}\Big]\\
& =\int_0^t\sum_{l=1}^{n-1}\binom{n}{l}\exp\Big(\int_0^{t_1}\xi(\zeta^{\ssup 1}_r)\,\d r\Big)\,\xi_2(\zeta^{\ssup 1}_{t_1})\,m_l(t-t_1,\zeta^{\ssup 1}_{t_1})\,m_{n-l}(t-t_1,\zeta^{\ssup 1}_{t_1})\,\d t_1.
\end{aligned}
\end{equation}
The induction hypothesis for $m_l$ and $m_{n-l}$ says that 
\begin{equation}\label{proofeq5}
\begin{aligned}
m_l(t-t_1,&\zeta^{\ssup 1}_{t_1})\,m_{n-l}(t-t_1,\zeta^{\ssup 1}_{t_1})\\
&=\sum_{k_1=0}^{l-1}\sum_{k_2=0}^{n-l}\sum_{T_{1}\in\mathcal{T}_{k_1}}\sum_{T_{2}\in\mathcal{T}_{k_2}}\sum_{I_1\in\mathcal{N}(T_{1})}\sum_{I_2\in\mathcal{N}(T_{2})} c_{k_1,l}\,c_{k_2,n-l}\,\Phi_{\zeta^{\ssup 1}_{t_1}}(T_{1},I_1,t-t_1)\Phi_{\zeta^{\ssup 1}_{t_1}}(T_{2},I_2,t-t_1).
\end{aligned}
\end{equation}
Let us denote by $T^{\ssup {1,2}}$ the tree in $\mathcal{T}_{k_1+k_2+1}$ formed by attaching the tree $T_{1}$ to the left of the unique child of the root and the tree $T_{2}$ to the right of the unique child of the root. Then the Markov property at time $t_1$ gives the following concatenation property of $\Phi$:
$$
\begin{aligned}
\sum_{I_1\in\mathcal{N}(T_{1})}&\sum_{I_2\in\mathcal{N}(T_{2})}\int_0^t \bq\Big[\exp\Big(\int_0^{t_1}\xi(\zeta^{\ssup 1}_r)\,\d r\Big)\,\Phi_{\zeta^{\ssup 1}_{t_1}}(T_{1},I_1,t-t_1)\Phi_{\zeta^{\ssup 1}_{t_1}}(T_{2},I_2,t-t_1)\Big]\,\d t_1\\
&=\sum_{I\in \mathcal{N}(T^{\ssup {1,2}})}\Phi_x(T^{\ssup {1,2}},I,t).
\end{aligned}
$$
Then, integrating both sides of (\ref{proofeq4}) with respect to $\bq$, we get
\begin{equation}\label{proofeq6}
\begin{aligned}
\bq\big[A(t)\1_{\{T\in[0,t]\}}\big]&=
\bq\big[\mbox{l.h.s.~of }(\ref{proofeq4})\big]\\
&=\sum_{l=1}^{n-1}\binom{n}{l}\sum_{k_1=0}^{l-1}\sum_{k_2=0}^{n-l}\sum_{T_{1}\in\mathcal{T}_{k_1}}\sum_{T_{2}\in\mathcal{T}_{k_2}} \,c_{k_1,l}\,c_{k_2,n-l}\,\sum_{I\in \mathcal{N}(T^{\ssup {1,2}})}\Phi_x(T^{\ssup {1,2}},I,t).
\end{aligned}
\end{equation}
Let $\mathcal{T}_k^{k_1,k-k_1-1}$ denote the set of trees in $\Tcal_k$ such that the two subtrees of the child of $\emptyset$ lie in $\Tcal_{k_1}$ and $\Tcal_{k-k_1-1}$, respectively. By changing the order of the sum we get that the right-hand side of (\ref{proofeq6}) is equal to
\begin{equation}
\sum_{k=1}^{n-1}\sum_{k_1=0}^{k-1}\sum_{T\in\mathcal{T}_k^{k_1,k-k_1-1}}\sum_{l=k_1+1}^{n-(k-k_1)}\binom{n}{l}\,c_{k_1,l}\,c_{k-k_1-1,n-l}\,\sum_{I\in\mathcal{N}(T)}\Phi_x(T,I,t).
\end{equation}
By (\ref{combid}) we have
\begin{equation}
\sum_{l=k_1+1}^{n-(k-k_1)}\binom{n}{l}\,c_{k_1,l}\,c_{k-k_1-1,n-l}=c_{k,n}.
\end{equation}
This, together with (\ref{inter2}) finishes the proof of (\ref{MomentFormula}).
\end{proof}

\begin{remark}\label{MomFormGeneral} There are also versions of Theorem~\ref{thrmmoment} for more general branching mechanisms as proposed in Section~\ref{sec-moregeneral} above. Under the additional assumption that $\sum_{k\in\N}k^n \xi_k$ converges almost surely, one can extend Theorem~\ref{thrmmoment} to this setting. The main change in (\ref{MomentFormula}) is that the terms involving the $c_{n,k}$ and $\xi_2$ must be replaced by a term of the form
$$
\sum_{{\rm mark}} \prod_v \sum_{k=2}^\infty \big(k^{{\rm mark}(v)}-k\big)\xi_k(X^{\ssup v}_{t_{I(v)}}),
$$
where the marks are now taken from a more complex set than $\{1,\dots,n\}$. Since the formulas arising are much more cumbersome, we abstained from writing them down carefully and proving them. However, it is easily seen from the above proof that they have a form which also admits an analysis of the large-$t$ limit of the moments of $m_n$ in the same way as we do in Section~\ref{sec-proofThm}.
\end{remark}

\section{Proof of the main result}\label{sec-proofThm}

\noindent In this section, we prove the main result of our paper, the moment asymptotics formulated in Theorem~\ref{thm-Main}. The proof will be crucially based on the Feynman-Kac-type formula for $m_n$ given in Theorem~\ref{thrmmoment} above. Another important ingredient is a large-deviations principle for the local times of the branching random walk, which we will provide in Section~\ref{sec-LDPBRW}. The proof of the lower and upper bound of the moment asymptotics are in Sections~\ref{sec-Prooflowbound} and ~\ref{sec-Proofuppbound}, respectively.

\subsection{LDP for the local times of the BRW}\label{sec-LDPBRW}

\noindent In this section we formulate and prove a large-deviations principle (LDP) for the normalised occupation time measures (the local times) of the BRW introduced in Section~\ref{sec-momentsBRWRE}, for a fixed tree $T=(\emptyset,S,L,E)\in\mathcal{T}_k$  and a fixed monotonous numbering $I \in \mathcal{N}(T)$, as the time parameter tends to infinity. 

We define the local times of the BRW as the sum of the times that its random walk segments $X^{\ssup{u,v}}$ with $(u,v) \in E$ spend in in a given site $z \in \Z^d$. More precisely, assume that $T\in\Tcal_k$ and let a time vector $\widehat t=(t_0,\dots,t_{k+1}) \in Z_k(t)$ be given and define the local time of the BRW in $z\in\Z^d$ as
    \begin{equation}  \label{BRW_local_times}
     \ell_{\widehat t}(z)=\sum_{(u,v) \in E}\int^{t_{I(v)}}_{t_{I(u)}} \delta_{X^{\ssup{u,v}}_r}(z) \,\d r. 
    \end{equation}
Then its total mass of is equal to 
\begin{equation}\label{totalmass}
m(\widehat{t}):=\sum_{z\in \Z^d} \ell_{\widehat{t}}(z)= \sum_{(u,v) \in E} (t_{I(v)}-t_{I(u)}).
\end{equation}
Hence, we normalise the local times and obtain
    \begin{equation}\label{BRW_occu_time_mes}
    L_{\widehat{t}}(z) =\frac{\ell_{\widehat{t}}(z)}{m(\widehat{t})}, \;\;\;z\in\Z^d;
    \end{equation}
a random element of the set $\Pcal(\Z^d)$ of all probability measures on $\Z^d$. Fix the starting site $x$ of the BRW. Let $\mathbb{T}_R^d$ be the lattice cube of length $2R+1$ centred at $x$. We consider the periodised local times
\begin{equation}
L_{\widehat{t}}^{\ssup R}(z):=\sum_{y\in (2R+1)\Z^d+x}L_{\widehat{t}}(z+y),\qquad z\in\mathbb{T}_R^d;
\end{equation}
a random element of the set $\Pcal(\mathbb{T}_R^d)$.
Our LDP reads as follows.

\begin{lemma}\label{thm-LDPBRW} Fix $k\in\N_0$, $T\in\Tcal_k$, $I\in\Ncal(T)$ and $R\in\N$. Furthermore, fix the starting site $x$ of the BRW and a vector $\widehat s=(s_0,\dots,s_{k+1})\in Z_k(1)$ and a sequence $\widehat s_t\to\widehat s$ as $t\to\infty$. Then the normalised local times $L_{t\widehat s_t}^{\ssup R}$ satisfy, as $t\to\infty$, the (full) large deviation principle with scale $tm(\widehat s)$ and rate function $\kappa S_R^{\ssup {\rm per}}$, where
\begin{equation}
S_R^{\ssup {\rm per}}(\mu)= \sum_{y_1,y_2 \in \mathbb{T}_R^d\colon y_1 \sim y_2} \Big(\sqrt{\mu(y_1)}-\sqrt{\mu(y_2)}\Big)^2. \label{ratefctper}
\end{equation} 
\end{lemma}

\begin{proof} The special case $k=0$ is classic and well-known, see \cite{DV75-83, G77, GM98}. Here $\widehat s=(0,1)$ and $m(\widehat s)=1$, and the BRW consists of just one random walk with start in $x$ and time interval $[0,t]$. This LDP holds even locally uniformly in $\widehat s\in Z_k(1)$, as is seen from the proof, which uses an eigenvalue expansion and the G\"artner-Ellis theorem, to turn it into modern notation. This also shows that the LDP is the same under the sub-probability measure that conditions on a fixed starting site and restricts to a fixed terminal site.

The general case is an easy consequence of that classical result, as the random walk segments $X^{\ssup{u,v}}$ with $(u,v)\in E$ are conditionally independent, after conditioning all the starting site and restricting to all the terminating sites, and these are only finitely many. Under this sub-probability measure, the normalised local times of each segment $X^{\ssup{u,v}}$ satisfy the LDP with scale $t (s_{I(v)}-s_{I(u)})$ and the rate function in \eqref{ratefctper}, and $L_{t\widehat s_t}^{\ssup R}$ is just an elementary finite convex combination of these independent objects. The claimed LDP follows by summing over all the starting and terminating sites, as these are only finite sums.
\end{proof}

\subsection{Proof of the lower bound}\label{sec-Prooflowbound}

\noindent Jensen's inequality gives, for any $n,p\in\N$ and any $x,y\in\Z^d$,
\begin{equation}\label{Jensen}
\langle m_n^p(t,x) \rangle \geq \langle m_n^p(t,x,y) \rangle = \langle E_x(\eta(t,y)^n)^p \rangle \ge \langle E_x(\eta(t,y))^{np} \rangle =  \langle m_1^{np}(t,x,y)\rangle,\qquad t>0.
\end{equation}
Now we can apply the result \eqref{MomentGM98} from \cite{GM98} for $np$ instead of $p$ and obtain the lower bound in \eqref{Moment_n_p}; see also \eqref{Moment_n_p_heur}. Here we recall that the logarithmic asymptotics of  $\langle m_1^{np}(t,x,y)\rangle$ are the same as for $\langle m_1^{np}(t,x)\rangle$, as mentioned in subsection \ref{sec-PAMconn}.

\subsection{Proof of the upper bound}\label{sec-Proofuppbound}

\noindent Now we give the corresponding upper estimate for the moments $ \langle m^p_n(t,x) \rangle $ (which applies then certainly also to $ \langle m^p_n(t,x,y) \rangle $). Recall from \eqref{MomentFormula} that $ \langle m^p_n(t,x) \rangle $ is the expectation of the $p$-th power of the sum of  $c_{k,n}\Phi_x(T,I,t)$ over $k\in\{0,\dots,n-1\}$,  $T\in\Tcal_k$ and $I\in\Ncal(T)$, where $\Phi_x(T,I,t)$ is given in \eqref{PhixTIdef}. Rewriting $\Phi_x(T,I,t)$ using the local times of the BRW,
\begin{equation}
\Phi_x(T,I,t)=\int_{Z_k(t)} \d\widehat t\, \E^{\ssup{T,I,\widehat t}}_{x} \Big[\exp\Big(\sum_{z\in \Z^d}\;\; \xi(z) \ell_{\widehat{t}}^{\ssup{T,I}}(z) \Big) \Big(\prod_{v\in S}\xi_2(X^{\ssup {u,v}}_{t_{I(v)}})\Big)\Big],
\end{equation}
it is clear (in the case $\rho=\infty$, from Assumption~\ref{momentassump}) that it is only the exponential term involving the local times that will turn out to be responsible for the claimed asymptotics, and only the summand for $k=n-1$ will turn out to give the leading asymptotics. We will dig this term out with repeated applications of Jensen's and H\"older's inequality.

We use the inequality
    \begin{equation}\label{ineq}
    \Big\langle \Big(\sum_{i\in\Xcal}X_i\Big)^p \Big\rangle \le |\Xcal|^{p-1}\sum_{i\in\Xcal}\big\langle X_i^p\big\rangle,
    \end{equation}
derived from Jensen's inequality, three times to get
\begin{equation}
\begin{aligned}
\langle m^p_n(t,x) \rangle& \le n^{p-1} \sum_{k=0}^{n-1} |\mathcal{T}_k|^{p-1} \sum_{T \in \mathcal{T}_k} |\mathcal{N}(T)|^{p-1} \sum_{I \in \mathcal{N}(T)}c^p_{k,n}\big\langle \Phi^p_x(T,I,t) \big\rangle\\
&=\e^{o(t)}\sum_{k=0}^{n-1}\sum_{T \in \mathcal{T}_k}\sum_{I \in \mathcal{N}(T)} \big\langle\Phi^p_x(T,I,t)\big\rangle. \label{MomentIneq1}
\end{aligned}
\end{equation}
Hence, the only dependence on $t$ now sits in the last expectation, $\langle \Phi^p_x(T,I,t) \rangle$, and it is enough to show that this term satisfies the claimed asymptotics and that it is maximal for $k=n-1$. Using  \eqref{ineq} in integral form for the integral over $\widehat t$ and Jensen's inequality for the expectation, we see that
    \begin{equation*}
    \Phi^p_x(T,I,t) \le  |Z_k(t)|^{p-1} \int_{Z_k(t)} \d\widehat t\, \Big(\E^{\ssup{T,I,\widehat t}}_{x} \Big[\exp\Big(\sum_{z\in \Z^d}\;\; \xi(z) \ell_{\widehat{t}}(z) \Big) \Big(\prod_{v\in S}\xi_2(X^{\ssup {u,v}}_{t_{I(v)}})\Big)\Big]\Big)^p, \label{ExpectIneq1}
    \end{equation*}
where $|Z_k(t)|=t^k/k!$ denotes volume of  $Z_k(t)$. Since this volume term is negligible for the logarithmic asymptotics, we only have to concentrate on the integral. Making the change of variables $\widehat t=t \widehat s$, we see that 
    \begin{equation}
    \Phi^p_x(T,I,t) \le  \e^{o(t)}\int_{Z_k(1)} \d\widehat s\, \Big(\E^{\ssup{T,I,t\widehat s}}_{x} \Big[\exp\Big(\sum_{z\in \Z^d}\;\; \xi(z) \ell_{t\widehat{s}}(z) \Big) \Big(\prod_{v\in S}\xi_2(X^{\ssup {u,v}}_{t s_{I(v)}})\Big)\Big]\Big)^p. \label{ExpectIneq1a}
    \end{equation}

At this stage we separate the cases $0\leq \rho <\infty$ and $\rho=\infty$.

\underline{ $0\leq\rho<\infty$:}
Our next step is to take expectation with respect to the branching/killing environment and to separate the exponential term from the powers of the $\xi_2$ terms by means of H\"older's inequality. Fix $q,q'>1$ (later chosen in dependence on $t$) satisfying $\frac 1q+\frac 1{q'}=1$, then we have
\begin{equation}
\begin{aligned}
\Big\langle\Big(&\E^{\ssup{T,I,t\widehat s}}_{x} \Big[\exp\Big(\sum_{z\in \Z^d}\;\; \xi(z) \ell_{t\widehat{s}}(z) \Big) \prod_{v\in S}\xi_2(X^{\ssup {u,v}}_{ts_{I(v)}}) \Big]\Big)^p\Big\rangle \\
&\le \Big\langle\E^{\ssup{T,I,t\widehat s}}_{x} \Big[\exp\Big(q\sum_{z\in \Z^d}\;\; \xi(z) \ell_{t\widehat{s}}(z) \Big)\Big]^{p/q} \E^{\ssup{T,I,t\widehat s}}_{x} \Big[\prod_{v\in S}\xi^{q'}_2(X^{\ssup {u,v}}_{ts_{I(v)}})\Big]^{p/q'}\Big\rangle\\
&\leq \Big\langle\E^{\ssup{T,I,t\widehat s}}_{x} \Big[\exp\Big(q\sum_{z\in \Z^d}\;\; \xi(z) \ell_{t\widehat{s}}(z) \Big)\Big]^{p}\Big\rangle^{1/q}
\Big\langle\E^{\ssup{T,I,t\widehat s}}_{x} \Big[\prod_{v\in S}\xi^{q'}_2(X^{\ssup {u,v}}_{ts_{I(v)}})\Big]^{p}\Big\rangle^{1/q'},  \label{ExpectIneq2}
    \end{aligned} 
\end{equation}
where we used H\"{o}lder's inequality twice. Now we show that the second term in the above display is negligible. Recall that $|S|=k-1$ and that $\xi_2(x)$ is i.i.d.~in $x\in\mathbb{Z}^d$. Then using Jensen's inequality and Fubini's theorem we get
\begin{equation}\label{prodtermesti}
\begin{aligned}
\Big\langle\E^{\ssup{T,I,t\widehat s}}_{x} \Big[\prod_{v\in S}\xi^{q'}_2(X^{\ssup {u,v}}_{ts_{I(v)}})\Big]^{p}\Big\rangle^{1/q'}&\leq \Big[\E^{\ssup{T,I,t\widehat s}}_{x} \Big\langle\prod_{v\in S}\xi^{pq'}_2(X^{\ssup {u,v}}_{ts_{I(v)}})\Big\rangle\Big]^{1/q'}
\leq\Big\langle \xi_2^{p(k-1)q'}(0) \Big\rangle^{1/q'},
\end{aligned}
\end{equation}
where for the last inequality we used the fact that $\xi_2\geq 0$. Now using the inequality $x\leq \e^x$ for $x>0$ we get
\begin{equation}
\begin{aligned}\label{boundsecterm}
&\Big\langle \xi_2^{p(k-1)q'}(0) \Big\rangle^{1/q'}\leq \Big\langle \e^{p(k-1)q'\xi_2(0)} \Big\rangle^{1/q'}=\exp\Big\{\frac {1}{ q'}H_2\big((k-1)pq'\big)\Big\},
\end{aligned}
\end{equation}
where we recall that $H_2$ denotes the logarithmic moment generating function of $\xi_2(0)$. Now we pick $q=q_t=1+\eps_t$ and $q_t'=1+1/\eps_t$ depending on $t$ such that $\eps_t\searrow0$ as $t\to\infty$. In our case where $\rho<\infty$, we have $H_2(t)\leq \rho t\log t+O(t)$ as $t\to\infty$, so it is clear that we can choose $\eps_t$ converging to 0 slowly enough so that as $t\to\infty$
\begin{equation}\label{qchoice}
\frac{\eps_t}tH_2(1/\eps_t)\to 0.
\end{equation}
Hence, by \eqref{qchoice} and \eqref{boundsecterm} we can conclude that the right-hand side of \eqref{prodtermesti} is $\e^{o(t)}$, i.e., the second term on the right-hand side of \eqref{ExpectIneq2} is negligible.

Proceeding as in the proof of \eqref{MomentGM98} in  \cite{GM98} and using the LDP of Lemma~\ref{thm-LDPBRW}, we get that
\begin{equation}
\Big\langle\E^{\ssup{T,I,t\widehat s}}_{x} \Big[\exp\Big(q_t\sum_{z\in \Z^d}\;\; \xi(z) \ell_{t\widehat{s}}(z) \Big)\Big]^{p}\Big\rangle^{1/q_t}\leq \exp\Big(\frac{1}{q_t}H\big(q_t pt m(\widehat{s})\big)-\frac{1}{q_t}2d\kappa pt m(\widehat{s}) \chi(\rho/\kappa)+o(t)\Big).
\end{equation}
Recall that the LDP in Lemma \ref{thm-LDPBRW} holds even uniformly in $\widehat{s}\in Z_k(1)$ away from zero. Hence, we can easily conclude that
\begin{equation}
\langle\Phi^p_x(T,I,t)\rangle^p \leq \e^{o(t)}\int_{Z_k(1)}\d\widehat{s}\,\exp\Big(\frac{1}{q_t}H\big(q_t pt m(\widehat{s})\big)-\frac{1}{q_t}2d\kappa pt m(\widehat{s}) \chi(\rho/\kappa)\Big).
\end{equation}
By \eqref{AssumptionHl}, for $\rho>0$ we have $H(t)\gg t$ as $t\to\infty$ and for $\rho=0$, $\chi(0)=0$. Finally, since $q_t\to 1$ as $t\to\infty$, by Laplace's method we get that, for any $T,I$ and $k\in\{0,\dots,n-1\}$,
\begin{equation}
\langle\Phi^p_x(T,I,t)\rangle^p \leq \e^{o(t)} \exp\Big(\frac{1}{q_t}H(q_t pt(k+1))-\frac{1}{q_t}2d\kappa p t (k+1) \chi(\rho/\kappa)\Big),
\end{equation}
as the main contribution comes from $\widehat s=(0,\dots,0,1)$, having total live time $m(\widehat s)=k+1$. The interpretation is that the main contribution to the moments comes from the BRW splitting into $k$ particles practically immediately after the beginning.
 
Hence, using \eqref{MomentIneq1} and the fact that $H(t)\gg t$ for $\rho>0$ and $\chi(\rho)=0$ for $\rho=0$ once again, by Laplace's method (\lq the largest rate wins\rq) we get 
\begin{equation}
\langle m^p_n(t,x) \rangle \leq \e^{o(t)}\exp\Big(\frac{1}{q_t}H(q_t pt n)-\frac{1}{q_t}2d\kappa p t n \chi(\rho/\kappa)\Big).
\end{equation}
The proof of the upper bound for $0\leq \rho <\infty$ is therefore finished by noting that
\begin{equation}\label{prooffinish}
\exp\Big(\frac{1}{q_t}H(q_t pt n)\Big)=\e^{H(ptn)} \e^{o(t)},\qquad t\to\infty.
\end{equation}
This is seen as follows. Recall that $q_t\searrow 1$, and note from \cite[Remark 1.1(b)]{GM98} that the convergence in \eqref{AssumptionHl} is uniform on $[0,1]$ and hence also locally uniform on $[0,\infty)$. Hence, writing
\begin{equation}
\frac 1t\Big(\frac1{q_t}H\big(q_t tpn\big)-H(tpn)\Big)=\frac1{q_t} \Big(\frac{H\big(q_t tpn\big)-pnq_tH(t)}{t}-\frac{H(pnt)-pnH(t)}{t}\Big)
+\frac{q_t-1}{q_t}\frac{pnH(t)-H(tpn)}t,
\end{equation}
shows that \eqref{prooffinish} holds and finishes the proof.

\medskip

\underline{$\rho=\infty$:} We start from \eqref{ExpectIneq1a}. In order to express the $p$-th power of the expectation, we introduce $p$ independent copies $X_{t_{I(u)}}^{\ssup{i,u,v}}$, $i=1,\dots,p$, of $X_{t_{I(u)}}^{\ssup{u,v}}$ and denote by $\ell^{\ssup p}_{t\widehat s}$ the sum over $i\in\{1,\dots,p\}$ of the local times of these random walks. For $z\in\Z^d$ define
\begin{equation}
r(z)=\sum_{i=1}^p\sum_{(u,v)\in S}\delta_z\big(X_{t_{I(u)}}^{\ssup {i,u,v}}\big);
\end{equation}
and introduce the notation
\begin{equation}\label{h2e}
G_2(l,k)=\frac{\big\langle \e^{l\xi_2(0)}\xi_2(0)^k\big\rangle}{\big\langle \e^{l\xi_2(0)}\big\rangle}=\frac{\big\langle \e^{l\xi(0)}\xi_2(0)^k\big\rangle}{\big\langle \e^{l\xi(0)}\big\rangle},\qquad l,k\in[0,\infty),
\end{equation}
where the last step used that $\xi_2(0)$ and $\xi_0(0)$ are independent; recall that $\xi=\xi_2-\xi_0$. From (\ref{ExpectIneq1a}) we have
\begin{equation}\label{rhoinfinity1}
 \begin{aligned}
\Big\langle\Phi^p_x(T,I,t)\Big\rangle&\leq \e^{o(t)}\int_{Z_k(1)} \d\widehat s\,\Big\langle\mathbb{E}\Big[\exp\Big(\sum_{z\in\Z^d}\ell_{t\widehat{s}}^{\ssup p}(z)\xi(z)\Big)\prod_{z\in\Z^d}\xi_2(z)^{r(z))}\Big]\Big\rangle\\
&=\e^{o(t)}\int_{Z_k(1)} \d\widehat s\, \mathbb{E}\Big[\Big\langle\exp\Big(\sum_{z\in\Z^d}\ell_{t\widehat{s}}^{\ssup p}(z)\xi(z)\Big)\prod_{z\in\Z^d}\xi_2(z)^{r(z))}\Big\rangle\Big]\\
&= \e^{o(t)}\int_{Z_k(1)} \d\widehat s\, \mathbb{E}\Big[\prod_{z\in\Z^d}\Big\langle \e^{\xi(0) \ell_{t\widehat s}^{\ssup  p}(z)}\xi_2(0)^{r(z)}\Big\rangle\Big]\\
&=\e^{o(t)}\int_{Z_k(1)} \d\widehat s\, \mathbb{E}\Big[\Big(\prod_{z\in\Z^d} \e^{H( \ell_{t\widehat s}^{\ssup  p}(z))}\Big)\prod_{z\in\Z^d}G_2\big(\ell_{t\widehat s}^{\ssup  p}(z), r(z)\big)\Big],
\end{aligned}
\end{equation}
where we used Fubini's theorem and (\ref{h2e}).

Note that $k\mapsto G_2(l,k)$ is log-convex for any $l$, since it is a moment generating function. Since $G_2(l,0)=0$, it is also easily seen to be log-subadditive. As a consequence, $\partial_l G_2(l,k)=G_2(l,k+1)-G_2(l,1)G(l,k)$ is nonnegative, hence $l\mapsto G_2(l,k)$ is increasing. Hence, since the local times $\ell_{t\widehat s}^{\ssup  p}(z)$ sum up to $tpm(\widehat{s})$, we have
$$
\prod_{z\in\Z^d}G_2\big(\ell_{t\widehat s}^{\ssup  p}(z), r(z)\big)
\leq \prod_{z\in\Z^d}G_2\big(tpm(\widehat{s}), r(z)\big)
\leq G_2\Big(tpm(\widehat{s}),\sum_{z\in\Z^d} r(z)\Big)=G_2\big(tpm(\widehat{s}),(n-1)p\big),
$$
since the $r(z)$ sum up to $p|S|=p(n-1)$. Note that the right-hand side is $\leq \e^{o(t)}$ by Assumption \ref{momentassump}. Using this fact in \eqref{rhoinfinity1}, we see that
$$
\Big\langle\Phi^p_x(T,I,t)\Big\rangle\leq\e^{o(t)}\int_{Z_k(1)} \d\widehat s\, \mathbb{E}\Big[\prod_{z\in\Z^d} \e^{H( \ell_{t\widehat s}^{\ssup  p}(z))}\Big].
$$
Now, precisely as in the proof of \eqref{MomentGM98} in \cite{GM98}, one proves that
$$
\mathbb{E}\Big[\prod_{z\in\Z^d} \e^{H( \ell_{t\widehat s}^{\ssup  p}(z))}\Big]\leq \e^{H(tpm(\widehat{s}))-2d\kappa tpm(\widehat{s})+o(t)}.
$$
An inspection of the proof, using the uniformity in the LDP in Lemma~\ref{thm-LDPBRW}, shows that this convergence is locally uniform in $\widehat s$. Hence, like in the above proof in the case $\rho<\infty$, we see that Laplace's method yields the result, after optimising over $\widehat s\in Z_k(1)$ and $k\in\{0,\dots,n-1\}$ (recall that $\chi(\infty)=1$).

\end{document}